\documentclass[]{scrartcl}
\usepackage[T1]{fontenc}
\usepackage[utf8]{inputenc}
\usepackage{amssymb}

\usepackage{amsmath,amsthm,amsfonts,amssymb}  
\usepackage{mathtools}
\usepackage{graphicx}
\usepackage{hyperref}
\usepackage{cleveref}
\usepackage{xcolor}

\newtheorem{theorem}{Theorem}[section]

\crefname{Proposition}{proposition}{propositions}
\Crefname{Proposition}{Proposition}{Propositions}

\crefname{Corollary}{corollary}{Corollaries}
\Crefname{Corollary}{Corollary}{Corollaries}

\crefname{Result}{result}{results}
\Crefname{Result}{Result}{Results}

\newtheorem{lemma}{Lemma}[section]
\crefname{Lemma}{lemma}{lemmas}
\Crefname{Lemma}{Lemma}{Lemmas}

\crefname{Remark}{remark}{remarks}
\Crefname{Remark}{Remark}{Remarks}

\newtheorem{definition}{Definition}[section]
\crefname{Definition}{definition}{definitions}
\Crefname{Definition}{Definition}{Definitions}

\title{Local times of self-intersection and sample path properties of Volterra Gaussian processes} 

\author{%
Olga~Izyumtseva
and 
Wasiur R. KhudaBukhsh\hspace{0.5mm}\href{https://orcid.org/0000-0003-1803-0470}{\includegraphics[width=3mm]{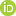}}
}

\newcommand{\Cov}{\mathsf{Cov}}
\newcommand{\Var}{\mathsf{Var}}
\newcommand{\E}{\mathsf{E}}
\newcommand{\Eof}[1]{\E\left(#1\right)}
\newcommand{\prob}{\mathsf{P}}
\newcommand{\probOf}[1]{\prob\left(#1\right)}

\newcommand{\ve}{\varepsilon}
\newcommand{\mbR}{{\mathbb R}}
\newcommand{\mbN}{{\mathbb N}}

\newcommand{\llim}[1]{\mathop{L_2\text{-}\lim}\limits_{#1}\,}
\newcommand{\myExp}[1]{\exp\left(#1\right)}
\newcommand{\differential}[1]{\mathrm{d}#1}
\newcommand{\inner}[2]{\left\langle #1, #2 \right\rangle}
\newcommand{\norm}[1]{\| #1 \|}

\begin{document}
\maketitle

\begin{abstract}
  We study a Volterra Gaussian process of the form $X(t)=\int^t_0K(t,s)\differential{W(s)},$ where $W$ is a Wiener process and $K$ is a continuous kernel. In dimension one, we prove a law of the iterated logarithm, discuss  the existence of local times and verify a continuous dependence between the local time and the kernel that generates the process. Furthermore, we prove the existence of the Rosen renormalized self-intersection local times for a planar Gaussian Volterra process.
\end{abstract}

\section{Introduction}
\label{sec:introduction}
Volterra Gaussian processes form an important class of stochastic processes with a wide range of applications. For example, they have been used for modelling velocity fields in turbulence, growth of cancer tumours \cite{BarndorffNielsen2007Ambit}, \cite{BarndorffNielsen2004Spatiotemporal}, and energy markets \cite{BarndorffNielsen2014Electricity,BarndorffNielsen2013Energy}. They also appear in the literature on infectious disease epidemiology as the limit of appropriately scaled population counts, often obtained by means of Functional Central Limit Theorems (FCLTs)  \cite{KhudaBukhsh2022FCLT,Pang2022FunctionalNonMarkovian,Pang2023Functional}. Volterra Gaussian processes were introduced by P. L\'evy in the 1950s. In \cite{Levy1956Special}, L\'evy presented the canonical representation for a given Gaussian process $X$, which is an integral representation of the process in terms of a Wiener process $W$ and a non-random Volterra kernel $K$ such that for each $t$ the represantation
\begin{equation}
\label{eq0.1}
X(t)=\int^t_0K(t,s)\differential{W(s)}
\end{equation}
holds strongly and the processes  $X$ and $W$ generate the same filtration.  L\'evy showed that the canonical representation of any Gaussian process is uniquely determined if it exists. T. Hida in \cite{Hida1960Canonical} systematically established L\'evy's theory  as well as proved some necessary and sufficient conditions for the existence of the canonical representation. Ito further provided conditions for the existence and the uniqueness of solutions of stochastic differential equations (SDEs) of the form \eqref{eq0.1} in \cite{Ito1979existence}.

The most well known examples of Volterra Gaussian processes are the Wiener process with  $K(t,s)=1_{[0,t]}(s),$ the Brownian bridge with $K(t,s)=(I-P)1_{[0,t]}(s),$ where $I$ is the identity operator and $P$ is a projection onto the linear subspace generated by the function $1_{[0,1]}$, the Ornstein--Uhlenbeck process defined on the real line with $K(t,s)=e^{-(t-s)}1_{[-\infty,t]}(s)$, and
the fractional Brownian motion with Hurst index $H\in (0,1)$ for which $K(t,s)$ is defined as follows
\begin{align*}
  K(t,s) &{} = \begin{cases}
    a_Hs^{\frac{1}{2}-H}\int^t_s(u-s)^{H-\frac{3}{2}}u^{H-\frac{1}{2}}\differential{u} & \text{when}\   H>\frac{1}{2},\\
    b_H\left[\left(\frac{t}{s}\right)^{H-\frac{1}{2}}(t-s)^{H-\frac{1}{2}}-\left(H-\frac{1}{2}\right)s^{\frac{1}{2}-H}\int^t_s(u-s)^{H-\frac{1}{2}}u^{H-\frac{3}{2}}\differential{u}\right]& \text{when}\  H<\frac{1}{2},
    \end{cases}     
\end{align*}
for $t>s$ (with $K(t,s)=0$ otherwise), and 
for some positive constants $a_H,\ b_H$. See also \cite{Biagini2008Stochastic,Pardoux1990Stochastic} for more details and other interesting examples.  Volterra-type Ornstein--Uhlenbeck processes were studied in \cite{Pham2018Volterra}. 
%
Volterra-type Gaussian processes are also closely related to the problem of optimal linear nonstationary filtering problem elaborately investigated by Kalman and Bucy (see \cite[Chapter 10]{Liptser1977Statistics}, \cite[Chapter 6]{Oksendal2003Stochastic}, \cite[Chapter 9]{Xiong2008introduction}).

Our aim in this paper is to study the asymptotic and geometric properties of Volterra Gaussian processes for which we provide a precise definition below. Let  $W(t),\ t\in [0,1]$, be a one-dimensional Wiener process. 

\begin{definition}
\label{def:Volterra} A centred Gaussian process $X(t),\ t\in[0,1]$, is called a Volterra Gaussian process, if for each $t\in(0,1]$, it admits the representation
$$
X(t)\stackrel{law}{=}\int^t_0K(t,s)\differential{W(s)},
$$
where $K:[0,1]^2\to\mbR$ is a Volterra kernel, i.e., 
$K(t,s)=0$  for all $s>t$, and $K$ satisfies 
$$
\sup_{t\in[0,1]}\int^t_0K(t,s)^2 \differential{s}<\infty. 
$$
\end{definition}

\subsection{Our contributions}
In this paper,  we consider Volterra Gaussian processes generated by Volterra kernels of the form 
$K(t,s)1_{[0,t]}(s),\ t,s\in[0,1],$ where $K$ is a continuous function on $[0,1]^2,$ continuously differentiable with respect to $t$ and $K(0,0)\neq 0.$ This is interesting since the Volterra Gaussian process then in some sense behaves similar to the Wiener process and inherits asymptotic and geometric properties of the Wiener process. The essential difference is that in general Volterra Gaussian processes are neither Markov processes nor martingales. Hence, the commonly used tools for studying asymptotic and geometric characteristics of Wiener processes are not applicable to the Volterra Gaussian processes we consider.  Our approach is based on the white noise representation of Gaussian processes. It allows us to study functionals of white noise via its Fourier-Wiener transform.

We prove that Volterra Gaussian processes generated by continuously differentiable kernels are integrators (in Section \ref{sec:Gaussianintegrators}). We then establish the law of the iterated logarithm (in Section \ref{sec:LIL}). Finally, we  discuss the existence of local times and prove the existence of the Rosen renormalized self-intersection local times for a planar Volterra Gaussian process in Section \ref{sec:rosen_local_times}. 

Even though we present the results for  a one-dimensional Volterra Gaussian process, the methods we use can be extended to higher dimensions in a fairly straightforward manner, albeit with more tedious notations. Since the results for higher dimensions do not seem to bring significantly new mathematical insights, we do not present them here and focus solely on the one-dimensional case for the sake of simplicity. 

For the purpose of studying local times and self-intersection local times of Volterra Gaussian processes, we apply the  approach developed in \cite{Izyumtseva2014Onthelocaltime,Izyumtseva2016moments,Dorogovtsev2015Properties,Dorogovtsev2011regularization}. Since Volterra Gaussian processes are often used in models of infectious disease epidemiology to approximate appropriately scaled population counts  in the limit of a large configuration model random graph \cite{KhudaBukhsh2022FCLT} or a well mixed population \cite{Pang2022FunctionalNonMarkovian,Pang2023Functional}, local times provide a useful means to assess the quality of such approximations since they can be used to study how much time the process spends near zero. Therefore, we believe our results are not only interesting from a theoretical perspective, but also could be useful in applications.  We refer the reader to the excellent books \cite{Marcus1999Renormalized,Marcus2006MarkovProcesses} for a detailed discussion on local times and renormalized self-intersection local times for Gaussian processes.

\subsection{Notational conventions}
Throughout this paper, we will use $X(t),\ t\in[0,1],$ to denote a one-dimensional Volterra Gaussian process generated by a kernel $K(t,s)1_{[0,t]}(s),\ t,s\in[0,1],$ as specified in Definition~\ref{def:Volterra}. We will use $x(t),\ t\in[0,1]$ to denote a general  Gaussian process. We will use $W(t),\ t\in[0,1],$ to denote a one-dimensional Wiener process, while $\mathcal{W}(t),\ t\in[0,1],$ will denote a planar Wiener process. We will use $\xi$ to denote a white noise in $L_2([0,1])$ generated by $W$ and we will use $\xi_1,\ \xi_2,$ to denote independent white noises in $L_2([0,1])$ generated by coordinates $W_1(t),\ W_2(t),\ t\in[0,1]$ of $\mathcal{W}(t),\ t\in[0,1].$   See Appendix~\ref{sec:white_noise} for the definition of white noise. We will use $\E$ to denote the expectation operator. We will use the notation $\llim{n\to\infty}$ to denote the limit as $n\to \infty$ in the $L_2$-sense. Inner products in a Hilbert space $H$ will be denoted by $\inner{\cdot}{\cdot}$. The norm of a vector will be denoted by $\norm{\cdot}$. The space of continuous functions from $A$ to $B$ will be denoted by $C(A,B)$. Given $e_1,\ldots,e_n$, elements of a Hilbert space $H$, we will use notation $\tilde{e}_1,\ldots,\tilde{e}_n$ for the orthogonal system of elements in $H$ obtained  from the elements $e_1,\ldots, e_n$  via the Gram--Schmidt procedure. The set of real numbers will be denoted by $\mbR$ and the set of natural numbers, by $\mbN$. Set $\mbN_0 = \mbN \cup \{0\}$. 

\section{Gaussian integrators}
\label{sec:Gaussianintegrators}



Gaussian integrators were introduced by A. Dorogovtsev in \cite{Dorogovtsev1998Stochastic}. They are called integrators because stochastic integral for a function from $L_2([0,1])$ with respect to the integrator is well defined. The original definition of an integrator is the following. 

\begin{definition}[Gaussian integrator \cite{Dorogovtsev1998Stochastic}]
\label{def:2} Let $x(t),\ t\in[0,1]$ be a centred Gaussian process with $x(0)=0.$  If there exists $c>0$ such that   the following estimate holds
\begin{equation}
\label{eq1.1}
\E\Big(\sum^{n-1}_{k=0}a_k(x(t_{k+1})-x(t_k))\Big)^2\leq c\sum^{n-1}_{k=0}a^2_k(t_{k+1}-t_k),
\end{equation}
for any $n\geq1$, and scalars $\ a_0, \ldots, a_{n-1}\in\mbR,$ $0=t_0<t_1<\ldots<t_n=1$,
then the process $x$ is said to be an integrator.
\end{definition} 
It can be verified that Wiener processes, Brownian bridges, and fractional Brownian motions with Hurst index $H\geqslant \frac{1}{2}$ are integrators. An integrator is associated with a continuous linear operator in $L_2([0,1])$  by means of a white noise representation in the same space (see Appendix~\ref{sec:white_noise} for the definition of white noise and Lemma \ref{lem:1} in  Appendix~\ref{sec:integrators}).

Consider a Volterra Gaussian process $X(t)$, $t \in [0, 1]$, as defined in Definition~\ref{def:Volterra}. Let $\xi$ be a white noise in $L_2([0,1])$ generated by $W$.  Then, it is easy to see that the Volterra Gaussian process $X$ can be represented as 
$$
X(t)=\inner{K(t,\cdot)1_{[0,t]} }{ \xi},\ t\in[0,1], 
$$
where $\inner{\cdot}{\cdot}$ denotes the inner product in $L_2([0,1])$ (see Appendix~\ref{sec:white_noise}). 
The next statement describes conditions on the kernel $K$ under which a  Volterra Gaussian process $X$  is an integrator. 

\begin{theorem}
\label{thm:integrator} Assume that the kernel $K(t,s),\ t,s\in[0,1]$ is continuous and continuously differentiable with respect to $t$. Then, the process $X$ is an integrator.
\end{theorem}

We need the following lemma in order to prove Theorem~\ref{thm:integrator}. 
\begin{lemma}
  \label{lem:integrator_meansq_charactertization} A centred continuous in mean square Gaussian process $x(t),\ t\in[0,1]$ is an integrator if and only if there exists a constant $c>0$ such that for any continuously differentiable function $f$ on $[0,1]$ with $f(0)=f(1)=0$ the following relation holds
  \begin{equation}
  \label{eq1.4}
  \E\Big(\int^1_0x(t)f^{\prime}(t)\differential{t}\Big)^2\leqslant c\int^1_0f^2(t)\differential{t}.
  \end{equation}
\end{lemma}
\begin{proof}[Proof of Lemma~\ref{lem:integrator_meansq_charactertization}] 
 Let $\{\frac{k}{n},\ k=0,\ldots,n\}$ be a partition of the interval $[0,1].$   Assume that the  process $x$ is an integrator. It implies that for a function $f$ that satisfies the conditions of the lemma, the following estimate holds
  \begin{equation}
  \label{eq1.5}
  \E\Big(\sum^{n-1}_{k=0}f\Big(\frac{k}{n}\Big)\Big(x\Big(\frac{k+1}{n}\Big)-x\Big(\frac{k}{n}\Big)\Big)\Big)^2\leqslant c\sum^{n-1}_{k=0}f\Big(\frac{k}{n}\Big)^2\frac{1}{n}.
  \end{equation}
  Applying the Abel transformation
  $$
  \sum^n_{k=m}a_kb_k=a_nB_n-a_mB_{m-1}-\sum^{n-1}_{k=m}(a_{k+1}-a_k)B_k,
  $$
  where $n\geqslant m\geqslant 1,\ B_0=0,\ B_k=\sum^{k}_{i=1}b_i$ and $a_k,\ b_k,\ k\geqslant 1$ are two real sequences, we get 
  \begin{align*}
    \sum^{n-1}_{k=0}f\Big(\frac{k}{n}\Big)\Big(x\Big(\frac{k+1}{n}\Big)-x\Big(\frac{k}{n}\Big)\Big)=-\sum^{n-1}_{k=0}x\Big(\frac{k+1}{n}\Big)\Big(f\Big(\frac{k+1}{n}\Big)-f\Big(\frac{k}{n}\Big)\Big) \eqcolon - S_{x, f}^{(n)}. 
 \end{align*}  
 It is clear that $S_{x, f}^{(n)}$ converges to $S_{x, f} \coloneq \int^1_0x(t)f^{\prime}(t)\differential{t}$ in mean square as $n\to\infty.$ Hence, it follows from \eqref{eq1.5} that 
  \begin{align}
    \begin{aligned}
      \E\left( S_{x, f}^{(n)} -S_{x, f} +S_{x, f} \right)^2  
      &{} =  \E\left( S_{x, f}^{(n)} -S_{x, f} \right)^2 
       + 2 \E\left(S_{x, f}^{(n)} S_{x, f}\right)
      + \E\left(S_{x, f}^2\right) \\
      &{}\quad \leqslant c\sum^{n-1}_{k=0}f\left(\frac{k}{n}\right)^2\frac{1}{n}.
    \end{aligned}
    \label{eq1.6}
  \end{align}

  Passing to the limit as $n\to \infty$ in mean square in \eqref{eq1.6}, we  obtain \eqref{eq1.4}.  
  
  Approximating the step functions by functions that satisfy the conditions of the lemma, we can see that \eqref{eq1.4} implies that $x$ is an integrator.
  \end{proof}

We are now ready to prove Theorem~\ref{thm:integrator}.

\begin{proof}[Proof of Theorem~\ref{thm:integrator}]

Let $f:[0,1]\to\mbR$ be a continuously differentiable function with $f(0)=f(1)=0$. Then
\begin{align*}
    \E\left(\int^1_0X(t)f^{\prime}(t)\differential{t}\right)^2&{} =\int^1_0\int^1_0f^{\prime}(t_1)f^{\prime}(t_2)\int^{t_1}_0 K(t_1,s)K(t_2,s)1_{[0,t_2]}(s)\differential{s} \differential{t_1}\differential{t_2}\\
    &{} = -\int^1_0\int^1_0f^{\prime}(t_2)f(t_1)\int^{t_2}_0K^{\prime}(t_1,s)K(t_2,s)1_{[0,t_1]}(s) \differential{s}\differential{t_1}\differential{t_2} \\
    &{}\quad -\int^1_0f^{\prime}(t_2)\int^{t_2}_0f(t_1)K(t_1,t_1)K(t_2,t_1)\differential{t_1}\differential{t_2} \\
    &{} = \int^1_0\int^1_0f(t_1)f(t_2)\int^{t_1\wedge t_2}_0K^{\prime}(t_1,s)K^{\prime}(t_2,s) \differential{s}\differential{t_1}\differential{t_2}\\
    &{} \quad +\int^1_0\int^{t_1}_0f(t_1)f(t_2)K^{\prime}(t_1,t_2)K(t_2,t_2)\differential{t_2}\differential{t_1}\\
    &{} \quad+\int^1_0\int^{t_2}_0f(t_2)f(t_1)K(t_1,t_1)K^{\prime}(t_2,t_1)\differential{t_1}\differential{t_2}\\
      &{} \quad +\int^1_0f^2(t_2)K^{2}(t_2,t_2)\differential{t_2}\\
     &{} \leqslant c\int^1_0f^2(t)\differential{t},
\end{align*}
where we performed integration by parts first with respect to $t_1$ in the second line and then with respect to $t_2$ in the third line. Here $c$ is a positive constant. Invoking Lemma~\ref{lem:integrator_meansq_charactertization} we conclude that $X$ is an integrator.



\end{proof}
  
In the light of Lemma \ref{lem:1}  and Theorem \ref{thm:integrator},
if for a Volterra Gaussian process $X(t),\ t\in[0,1]$
the kernel $K(t,s),\ t,s\in[0,1]$ is continuous and  continuously differentiable with respect to $t$, then there exists a continuous linear operator $A$ on $L_2([0,1])$ such that
$$
X(t)=\inner{A1_{[0,t]}}{\xi},
$$
where $\xi$ is a white noise in $L_2([0,1])$ generated by the Wiener process $W.$ Moreover, for any $f\in L_2([0,1])$, the extended stochastic integral with respect to $X$ is well defined and is given by
$$
\int^1_0f(t)\differential{X(t)}=\int^1_0Af(t)\differential{W(t)}.
$$
Note that the kernels of Volterra Gaussian processes that generate integrators belong to the singular case of kernels described in \cite{Alos2001Stochastic} for which the stochastic calculus was developed.

\section{Law of the iterated logarithm}
\label{sec:LIL}
In this section, we will assume that  $K(t,s),\ t,s\in[0,1]$ is a continuous kernel, continuously differentiable with respect to $t$  with $K(0,0)\neq 0.$ 
It implies that there exists a positive constant $L$ such that 
\begin{equation}
\label{eq6.0}
\sup_{s\in[0,1]} |K(t_2, s) - K(t_1, s) | \leqslant L|t_2-t_1|,
\end{equation}
for any $t_1,t_2\in[0,1]$. 
Then, for $s<t,$ we have
$$
\E(X(t)-X(s))^2=\int^t_sK^2(t,r)\differential{r} +\int^s_0(K(t,r)-K(s,r))^2 \differential{r}.
$$
Hence it follows from \eqref{eq6.0} that there exists a positive constant $L_1$ such that
$$
\E(X(t)-X(s))^2\leqslant L_1|t-s|.
$$
Applying the Kolmogorov's continuity theorem, one can conclude that there exists a continuous modification of the process $X$. We are now ready to state and prove a law of the iterated logarithm for the Volterra Gaussian process $X$.

\begin{theorem}
\label{thm:LIL} 
Let the kernel  $K(t,s),\ t,s\in[0,1]$ that generates  a Volterra Gaussian process $X$ be continuous and continuously differentiable  with respect to $t$ with $K(0,0)\neq 0$. Then,  
$$
\limsup_{t\to 0}\frac{X(t)}{\sqrt{2\E\left(X(t)^2\right)\log\log \frac{1}{t}}}=1,
$$
almost surely (a.s.).

\end{theorem}
\begin{proof}[Proof of Theorem~\ref{thm:LIL}] Let us define the function 
$$
h(t)=\sqrt{2\E\left(X(t)^2\right)\log\log \frac{1}{t}}.
$$
Now, fix $\ve>0$ and $q<1$. Consider the sequence of sets $\{A_n : n\ge 1\}$,  where $A_n$ is  given by 
$$
A_n=\Big\{\max_{t\in[0,q^n]}X(t)\geqslant (1+\ve)h(q^n)\Big\}.
$$
To get an upper estimate for $\prob(A_n)$, we 
apply an estimate for the probability tails of Gaussian extrema from \cite{Samorodnitsky1991tails}, which we state as Lemma~\ref{lem:6}  in Appendix~\ref{subsec:Gaussian tails} for the sake of completeness. Since $K$ is continuously differentiable with respect to $t$, we have
$$
\Eof{X(t)^2}=\int^t_0K^2(t,s)\differential{s}\sim K^2(0,0)t,
$$ 
as $t\to 0$.  It follows from Lemma~\ref{lem:6} that for any $\delta>0,\ \kappa>0 $ there exist a positive constant $k(\delta)<\infty$ and an integer $N$ such that for any $n\geqslant N$, we have 
\begin{align*}
  \prob(\max_{t\in[0,q^n]}X(t)\geqslant (1+\ve)h(q^n)) &{} \leqslant k(\delta)\exp\left({-\frac{(1-\delta)(1+\ve)^2(h(q^n))^2}{2(1+\kappa)K^2(0,0)q^n}}\right)\\ 
  &{} \leqslant k(\delta)\left(\frac{1}{n\log\frac{1}{q}}\right)^{c(\delta, \ve, \kappa)},
\end{align*}
where the constant $c(\delta, \ve, \kappa)$ is given by
\begin{align*}
  c(\delta, \ve, \kappa) &{} = \frac{(1-\delta)(1+\ve)^2(1-\kappa)}{1+\kappa}.\\
\end{align*}
Choosing $\delta<\left(1-\frac{1+\kappa}{(1+\ve)^2(1-\kappa)}\right)$ and
$\kappa<\frac{(1+\ve)^2-1}{(1+\ve)^2+1}$, we can ensure that $c(\delta, \ve, \kappa) >1$, which, in turn, implies 
$$
\sum^{\infty}_{n=1}\prob(A_n)<\infty.
$$
Therefore, it follows from the first part of the Borel--Cantelli lemma \cite{shiryaev2016probability1,Chandra2012Borel} that 
$$
\prob(A_n\ \text{i.o.})=0.
$$


Let $A=\{A_n\  \text{i.o.}\}^c.$ For any $\omega\in A$ and $q^{n+1}\leqslant t<q^n$, we have 
\begin{align}
  \begin{aligned}
    X(t)&{}\leqslant \sup_{t\in[0,q^n]}X(t)\leqslant (1+\ve)h(q^n)\leqslant q^{-\frac{1}{2}}(1+\ve) h(q^{n+1})\leqslant \frac{q^{-\frac{1}{2}}(1+\ve)(1+\kappa)}{1-\kappa}\ h(t).
  \end{aligned}
  \label{eq1.9}  
\end{align}
 It follows from \eqref{eq1.9} that 
$$
\limsup_{t\to0} \frac{X(t)}{h(t)}\leqslant 1\ \text{a.s.}
$$

Since we work with non-independent events, we will need a generalized version of the Borel--Cantelli lemma. For the sake of completeness, we provide a statement in Lemma~\ref{lem:GBCL} in Appendix~\ref{appendix:extra}. Consider the events  
 $$
 B_n=\{X(q^n)-X(q^{n+1})>h(q^{n}-q^{n+1})\}. 
 $$
 Then, there exists $N>0$ such that for $n\geqslant N$, we have 
 \begin{align}
 \begin{aligned}
\label{eq6.4}  
\prob(B_n)&{}=\prob\Big(Z\geqslant \frac{h(q^{n}-q^{n+1})}{\sqrt{\E(X(q^n)-X(q^{n+1}))^2}}\Big)\\
&{}\geqslant\prob\left(Z\geqslant \sqrt{\log\log\frac{1}{q^n-q^{n+1}}}\right)\\
&{}\geqslant k_1\frac{\exp\left({-\log\log\frac{1}{q^n-q^{n+1}}}\right)}{\sqrt{2\log\log\frac{1}{q^n-q^{n+1}}}}\\
&{}\geqslant k_2\frac{1}{n\log n},
\end{aligned}
\end{align}
where $Z$ is a standard Gaussian random variable, i.e., $Z\sim \mathcal{N}(0,1)$,  and the numbers $k_1, k_2$ are some positive constants.   Therefore, we have
 $$
 \sum^{\infty}_{n=1}\prob(B_n)=\infty.
 $$
Now, consider the collection of Gaussian vectors
  $$
 \eta_{n,m}=\left(\frac{X(q^n)-X(q^{n+1})}{\sqrt{\E(X(q^n)-X(q^{n+1}))^2}},\ \frac{X(q^m)-X(q^{m+1})}{\sqrt{\E(X(q^m)-X(q^{m+1}))^2}}\right),\ n,m\geqslant 1.
 $$
 It can be verified that the sequence $ \eta_{n,m}$  converges weakly  to $\mathcal{N}(0,I)$ as $n,m\to\infty.$  Consequently, we have
 $$
 \prob(B_n\cap B_m)\sim \prob(B_n)\prob(B_m),\ \text{ as } m,n\to\infty,
 $$
 i.e., for any $\ve>0,$ there exists $N>0$ such that for all $n,m\geqslant N$, we have 
 $$
  \prob(B_n\cap B_m)\leqslant (1+\ve)\prob(B_n)\prob(B_m).
  $$
  Therefore, it follows from Lemma~\ref{lem:GBCL} in Appendix~\ref{appendix:extra} that 
    $$
  \prob(B_n\ \text{i.o.})\geqslant\frac{1}{1+\ve}.
  $$
  Now, set $B=\{B_n\ \text{i.o.}\}.$ Then,  for $\omega\in B$, we have 
  \begin{equation}
\label{eq6.5} 
  X(q^n)\geqslant X(q^{n+1})+h(q^n-q^{n-1})
  \end{equation}
  holds for infinitely many $n$. From the the upper estimate we have 
  $$
  X(q^{n+1})\geqslant -2h(q^{n+1}).
  $$
    Hence, it follows that 
\begin{align}
\begin{aligned}
\frac{X(q^n)}{h(q^n)}&{}\quad \geqslant\frac{-2h(q^{n+1})+h(q^n-q^{n+1})}{h(q^n)}\\
 &{}\quad \geqslant-2\Big(\frac{1+\kappa}{1-\kappa}\Big)^{\frac{1}{2}}\sqrt{q}(1+\ve)+\Big(\frac{1+\kappa}{1-\kappa}\Big)^{\frac{1}{2}} \sqrt{1-q}(1-\ve),
 \end{aligned}
 \end{align}
 for some  $\kappa>0,\ \ve>0.$ Letting $q,\ \kappa,\ \ve$ go to zero, one gets
  $$
  \limsup_{n\to\infty}\frac{X(q^n)}{h(q^n)}\geqslant 1.
  $$
  Hence, for any $\alpha>0,$ we have
  $$
  \frac{1}{1+\alpha}\leqslant \probOf{\limsup_{t\to 0}\frac{X(t)}{\sqrt{2\Eof{X(t)^2}\log\log \frac{1}{t}}}=1}\leqslant 1,
  $$
  which implies that
  $$
  \limsup_{t\to 0}\frac{X(t)}{\sqrt{2\Eof{X(t)^2}\log\log \frac{1}{t}}}=1
  $$
  holds a.s., as required. 
  \end{proof}
 

\section{Local times and renormalized self-intersection local times}
\label{sec:rosen_local_times}
 In this section, we begin with a discussion on the existence of local times for a Volterra Gaussian process in $\mbR$. We also prove that the local time  depends continuously on the kernel that generates the process. Finally, we construct the Rosen renormalized self-intersection local time for a planar Volterra Gaussian process.

Let $x(t),\ t\in [0,1]$ be an $\mbR$-valued stochastic process. Let us try to measure how much time the process $x$ spends inside  small neighbourhoods of some $y\in\mbR$ up to time $t.$ We will denote this random variable by $l^x(t,y).$ 
Intuitively, one could define $l^x(t,y)$ as
  \begin{equation}
\label{eq6.50} 
  \int^t_0\delta_y(x(s))\differential{s},
 \end{equation}
  where $\delta_y$ is  the Dirac delta function at $y$. The formal expression \eqref{eq6.50} measures the amount of time $s\in[0,t]$ such that the values $x(s)$ are equal to $y.$
  To give a rigorous definition, consider the family of functions
 $$
  f_{\ve,y}(z)=\frac{1}{(2\pi\ve)^{\frac{1}{2}}}\myExp{-\frac{(z-y)^2}{2\ve}},\ z,\ y\in\mbR,\ \ve>0.
 $$
 Note that $f_{\epsilon,y}$  converges weakly to $\delta_y$ for each $y\in\mathbb{R}$ as $\ve \to 0$ in the sense that for any continuous bounded 
 function $\phi:\mathbb{R}\to\mathbb{R}$, we have 
 $$
 \int_{\mathbb{}R}\phi(z)f_{\ve,y}(z)dz\to \phi(y),
  $$
 as $\ve\to 0$. Now, let us define the family of approximations for $l^x(t,y)$ as follows
 $$
   l^x_{\ve}(t,y)=\int^t_0f_{\ve, y}(x(s))\differential{s}.
   $$

\begin{definition}
\label{def:local_time} The random variable
$$
l^x(t,y)=\llim{\ve\to0}l^x_{\ve}(t,y)
$$
is said to be the local time of the process $x$ at $y$ up to time $t$, when the limit exists.
\end{definition}

Let $X(t),\ t\in[0,1]$ be a Volterra Gaussian process generated by a continuous kernel $K$ as defined in Definition~\ref{def:Volterra}. The following lemma shows that the local time of the process $X$ exists at any point $y\in\mbR.$

\begin{lemma}
\label{lem:local_time} Assume that $\min_{t,s\in[0,1]}K(t,s)\neq 0,$ then the local time $l^{X}(t, y)$ of the process $X$ exists at any point $y\in\mbR$, for all $t\in [0,1]$. 
\end{lemma}

\begin{proof}[Proof of Lemma~\ref{lem:local_time}] In order to prove the lemma, we will use a sufficient condition proved by A. Rudenko in \cite{rudenko2006existence} for the existence of local times for Gaussian processes in $\mbR$. For the sake of completeness, a statement of the criterion is provided in Lemma~\ref{lem:Rudenko}. 


Note that, for $t>s$, we have
\begin{align}
\begin{aligned}
\det\Cov(X(s),X(t))&{}=\Var(X(s))\Var(X(t)-X(s)\mid X(s))\\
&{} \geqslant  \Var(X(s))\Var(X(t)-X(s)\mid W(s))\\
&{}= \int^s_0K^2(s,r)\differential{r}\int^t_sK^2(t,r)\differential{r}\geqslant c\ s(t-s),
\end{aligned}
\end{align}
where $c$ is a positive constant, and $\Cov(X(s), X(t))$ denotes the $2\times 2$ covariance matrix of the random vector $(X(s), X(t))$. Since
$$
\int^1_0\int^1_0\frac{1}{\sqrt{s(t-s)}}\differential{s}\differential{t}<\infty,
$$
by an application of Rudenko's criterion from Lemma~\ref{lem:Rudenko}, we conclude that the
local time of the process $X$ exists at zero on $[0,1],$ and consequently, it exists at any point $y\in\mbR.$
\end{proof}
We note that an alternative proof can be obtained by means of the \emph{local nondeterminism property} (see \cite{Nolan1989LND,Berman1989LNDgeneral,Berman1991Selfintersections}). Local nondeterminism  as a condition for the existence of local times for one-dimensional Gaussian processes was introduced by S. Berman in \cite{berman1973local}.  Let $x(t),\ t\in [0,1]$ be an $\mbR$-valued Gaussian process with mean $0$. The process $x$ is said to have the local nondeterminism property if there exists $d>0$ such that
 \begin{equation}
\label{eq6.7} 
\Eof{(x(t)-x(s))^2}>0, \text{for all}\  s, t\in (0,1),\ \text{ and }\  0<|t-s|\leq d, 
 \end{equation}
 \begin{equation}
\label{eq6.8} 
\text{ and } \Eof{x^2(t)}>0, \text{ for all }\  t\in (0,1).
\end{equation}
For Volterra Gaussian processes that satisfy conditions  \eqref{eq6.7} and  \eqref{eq6.8}, the condition of the local nondeterminism on $(0,1)$ reduces the following requirement on the kernel $K$:
 \begin{equation}
\label{eq6.9} 
 \lim_{c\downarrow 0}\inf_{0<t-s\leqslant c}\frac{\int^t_sK^2(t,r)\differential{r}}{\int^s_0(K(t,r)-K(s,r))^2 \differential{r}}>0.
  \end{equation}
 Intuitively, the condition  \eqref{eq6.9} means that a future observation cannot be predicted by past observations. An extension of the local nondeterminism property for the Volterra Gaussian processes in $\mbR^d$ was introduced in \cite{harang2022regularity} and defined as follows. If for a Volterra Gaussian process $x(t)=\int^t_0K(t,s)\differential{W(s)},\ t\in[0,1]$ the kernel satisfies the following condition
  \begin{equation}
\label{eq6.10} 
 \lim_{t\to0}\inf_{s\in(0,t]}\frac{1}{(t-s)^{\zeta}}\int^t_sK^2(t,r)\differential{r}>0,
  \end{equation}
then the process $x$ is said to be $(2,\zeta)$-locally nondeterministic. Moreover, it was proved in \cite{harang2022regularity} that for $(2,\zeta)$-locally nondeterministic Volterra Gaussian processes in $\mbR^d$ with some $\zeta\in (0,\frac{2}{d})$ local time exists at any point. Note that, a one-dimensional Volterra Gaussian process $X$ generated by a continuous kernel $K(t,s),\ t,s\in[0,1]$, which is continuously differentiable with respect to $t$ with $K(0,0)\neq 0$, is $(2,1)$-locally nondeterministic, which implies that the local time for $X$ exists at any point of real line. For $d\geqslant 2$ the condition \eqref{eq6.10} is not fulfilled for any $\zeta\in (0,\frac{2}{d}).$  

Before discussing renormalized self-intersection local times, we make a second remark, which is interesting on its own and does not seem to have been discussed in the literature previously. Let $X_n,\ n\geqslant 1$, and $ X$ be Volterra Gaussian processes generated by continuous kernels $K_n,\ n\geqslant 1$, and $ K$, respectively. With slight abuse of notation, let us define 
\begin{align}
  l^{(n)}(t, y):= l^{X_n}(t, y), \ \text{ and } l(t, y):= l^{X}(t, y). \label{eq:local_time_n}
\end{align}
The following lemma is about the continuous dependence of the local times for a Volterra Gaussian process on the kernels generating the processes.

\begin{lemma}\label{lem:continuous_dependence}
  Let kernels $K_n(t,s),\ t,s\in[0,1],\ n\geqslant1$, and $ K(t,s),\ t,s\in[0,1]$ be continuous, continuously differentiable with respect to $t$ with $K_n(0,0)\neq 0$ and $K(0,0)\neq 0.$  Furthermore, assume that 
  \begin{equation}
    \label{eq:continuous_dependence} 
      \max_{s,t\in[0,1]}|K_n(t,s)-K(t,s)|\to0,\ \text{as}\ n\to0 .
  \end{equation}
  Then, we have 
  \begin{align}
  \begin{aligned} 
   \label{eq:continuous_dependence1} 
  \lim_{n\to \infty} \max_{t\in[0,1]}\Eof{(X_n(t)-X(t))^2 } & = 0,
\end{aligned}
\end{align}
 \begin{align}
  \begin{aligned} 
   \label{eq:continuous_dependence2} 
\text{and}\quad \lim_{n\to \infty} \Eof{\int_{\mbR}(l^{(n)}(t, y)-l(t, y))^2 \differential{y}} & = 0,\ \text{ for all } t\in [0, 1].
\end{aligned}
\end{align}
\end{lemma}
\begin{proof}[Proof of Lemma~\ref{lem:continuous_dependence}]
The proof of \eqref{eq:continuous_dependence1} is straightforward. To prove \eqref{eq:continuous_dependence2}, let us check that for each $t\in[0,1]$, we have 
 \begin{align*} 
 &{}\Eof{\int_{\mbR}l^{(n)}(t, y)^2 \differential{y}}\to\Eof{\int_{\mbR}l(t, y)^2 \differential{y}}
    \end{align*}
and 
 \begin{align*} 
 &{}\Eof{\int_{\mbR}l^{(n)}(t, y)l(t, y) \differential{y}}\to\Eof{\int_{\mbR}l(t, y)^2 \differential{y}}
    \end{align*}
    as $n\to\infty.$
    Using the approximation procedure and the Fourier--Wiener transform, one can check that for each $s,t\in[0,1]$ the following relation holds
    $$
    \int_{\mbR}\delta_y(X_n(t))\delta_y(X_n(s))\differential{y}=\delta_0(X_n(t)-X_n(s)).
    $$
    It implies that
   \begin{align*}
    &{}\Eof{\int_{\mbR}l^{(n)}(t, y)^2 \differential{y}}=\Eof{\int^1_0\int^1_0\delta_0(X_n(t)-X_n(s))\differential{s}\differential{t}}\\
    &{}=\lim_{\ve\to0}\Eof{\int^1_0\int^1_0f_{\ve}(X_n(t)-X_n(s))\differential{s}\differential{t}}\\
         &{}=\frac{2}{2\pi}\int_{\Delta_2}\frac{1}{\sqrt{\|K_n(t,\cdot)1_{[s,t]}\|^2+\|K_n(t,\cdot)-K_n(s,\cdot)1_{[0,s]}\|^2}}\differential{s}\differential{t}, 
  \end{align*}
 which converges to
 $$
\Eof{\int_{\mbR}l(t,y)^2 \differential{y}}=\frac{2}{2\pi}\int_{\Delta_2}\frac{1}{\sqrt{\|K(t,\cdot)1_{[s,t]}\|^2+\|K(t,\cdot)-K(s,\cdot)1_{[0,s]}\|^2}}\differential{s}\differential{t}
 $$
 as $n\to\infty$, due to  Lebesgue's dominated convergence theorem. Applying the same arguments, one can check that
 \begin{align*} 
 &{}\Eof{\int_{\mbR}l^{(n)}(t, y)l(t, y) \differential{y}}\to\Eof{\int_{\mbR}l(t, y)^2 \differential{y}}
    \end{align*}
    as $n\to\infty$, which completes the proof of the lemma.
    \end{proof}



 Let us now investigate self-intersection local times. Let $$\mathcal{W}(t) \coloneqq (W_1(t), W_2(t)),\ t\in[0,1]$$ be an $\mbR^2$-valued (planar) Wiener process. It follows from the Dvoretzky, Erd\"os, Kakutani theorem \cite{dvoretzky1954multiple} that almost all paths of $\mathcal{W}$ have points of self-intersection of arbitrary finite multiplicity. So, by analogy to the local time, one can try to measure how much time a planar Wiener process spends in small neighbourhoods of its self-intersection points. For $k \in \mbN$, it can be measured by the following formal expression
\begin{equation}
\label{eq:selfinter_formal}
\int_{\Delta_k}\left(\prod^{k-1}_{i=1}\delta_0(\mathcal{W}(t_{i+1})-\mathcal{W}(t_i))\right)\prod^{k}_{i=1}\differential{t_i},
\end{equation}
where $\Delta_k=\{0\leqslant t_1\leqslant\ldots\leqslant t_k\leqslant 1\}$ and $\delta_0$ is the two-dimensional Dirac delta function, with $0$ denoting the vector $(0,0)$ in $\mbR^2.$ To give a rigorous meaning to  \eqref{eq:selfinter_formal}, in the same way as we did to define local times in Definition~\ref{def:local_time}, consider a family of approximations for  $\delta_0$ as 
$$
f_{\ve}(y):=f_{\ve,0}(y)=\frac{1}{2\pi\ve}\myExp{-\frac{\|y\|^2}{2\ve}},\ \ve>0,\ y\in\mbR^2, 
$$
and define approximations for the self-intersection local time in \eqref{eq:selfinter_formal} of $\mathcal{W}$ as follows
$$
T_{\ve,k}^{\mathcal{W}}=\int_{\Delta_k}\left(\prod^{k-1}_{i=1}f_\ve(\mathcal{W}(t_{i+1})-\mathcal{W}(t_i))\right) \prod_{i=1}^{k}\differential{t_i}.
$$
Similar to local times, we could define 
\begin{equation}
\label{eq:naive_self_loc_k}
T^{\mathcal{W}}_k=\llim{\ve\to0}T^{\mathcal{W}}_{\ve,k}, 
\end{equation}
and call it a self-intersection local time for $\mathcal{W}$ if the limit exists. However, unfortunately, it is not difficult to see that the limit in \eqref{eq:naive_self_loc_k} does not exist. The reason is that a planar Wiener process has too many self-intersections in the  neighbourhoods of the diagonals of $\Delta_k.$ To compensate the contributions of the diagonals, we need a renormalization. Renormalizations for the self-intersection local time of $\mathcal{W}$ were precisely formulated by S. R. S. Varadhan, J. Rosen and E. B. Dynkin in \cite{varadhan1969appendix,rosen1986renormalized,dynkin1988regularized}. In the present work, we adopt the Rosen renormalization defined as follows
$$
L_{\ve,k}^{\mathcal{W}}\coloneqq \int_{\Delta_k}\left(\prod^{k-1}_{i=1}\bar{f_{\ve}}(\mathcal{W}(t_{i+1})-\mathcal{W}(t_i))\right) \prod_{i=1}^{k}\differential{t_i},
$$
where we have used the notation $\bar{f_\ve}(\cdot)$ to mean  $\bar{f_\ve}(\cdot)=f_\ve(\cdot)-\Eof{f_\ve(\cdot)}$. 
In \cite{rosen1986renormalized}, Rosen proved that for any $k\geqslant 2$, the following random variable 
$$
L_k^{\mathcal{W}} \coloneq \llim{\ve\to0}L_{\ve,k}^{\mathcal{W}}
$$
exists. We refer to $L_k^{\mathcal{W}}$ as the Rosen renormalized self-intersection local time of multiplicity $k$ for the planar Wiener process $\mathcal{W}.$

Given the discussion above, one expects that a naive definition of self-intersection local times for Volterra Gaussian processes as the $L_2$-limit in \eqref{eq:naive_self_loc_k} does not work and we need to renormalize the integrand. In order to construct the Rosen renormalized self-intersection local time of multiplicity $k$ for a planar Volterra Gaussian process, we use the approach introduced in \cite{Dorogovtsev2011regularization} based on the white noise representation of Gaussian processes and by studying functionals of the white noise via its Fourier--Wiener transform \cite{Cameron1945examples,Cameron1945Fourier}. 
Our main aim is to apply  Rosen renormalization to the formal Fourier--Wiener transform of the $k$-multiple self-intersection local time for the planar Volterra Gaussian process.   Intuitively, one can expect that it could be done if the increments of $X$ on small time intervals behave like the increments of Wiener process. In order to ensure such behaviour of increments of $X$, we need to impose some conditions on the kernel $K$ generating the process. For that purpose, we need the notion of a strongly locally nondeterministic process introduced in \cite{Dorogovtsev2011regularization}, which we describe below. 


Let $x(t)=\inner{g(t)}{\xi},\ t\in[0,1]$ be a one dimensional centred Gaussian process, where $g\in C([0,1],L_2([0,1]))$, the space of continuous maps from $[0, 1]$ to $L_2([0, 1])$,  and
$\xi$ is a white noise in $L_2([0,1])$. Denote by $G(e_1,\ldots,e_n)$ the Gram determinant constructed by the elements $e_1,\ldots,e_n$.
Assume that the function $g$  is such that for any 
$0\leqslant t_1<\ldots<t_n\leqslant 1$, we have
$$
G(\Delta g(t_1),\ldots,\Delta g(t_{n-1}))>0,
$$
where $\Delta g(t_i)=g(t_{i+1})-g(t_i),\ i=1,\ldots,n-1.$

\begin{definition}[Strongly locally nondeterministic process \cite{Dorogovtsev2011regularization}]
\label{def:strongLND} The process $x$ (or equivalently, the function $g$) is said to be strongly locally nondeterministic if for any $k\geqslant 2$ and subset $M\subseteq \{1,\ldots, k-1\}$, we have 
$$
G(\Delta g(t_1),\ldots,\Delta g(t_{k-1}))\sim G(\Delta g(t_i),\ i\notin M)\prod_{i\in M}\|\Delta g(t_i)\|^2,
$$
when $\max_{i\in M}\Delta t_i\to0.$
\end{definition}
It is interesting to compare the strong local nondeterminism with the Berman's local nondeterminism \cite{berman1973local}. To do this, let us reformulate the Berman's local nondeterminism in our context.
\begin{definition}
\label{def:localLND} The process $x$ (or equivalently, the function $g$) is said to be locally nondeterministic if for every $k\geqslant 1$
$$
\liminf_{(t_n-t_1)\to0}\frac{G(\Delta g(t_1),\ldots,\Delta g(t_{n-1}))}{G(\Delta g(t_1),\ldots,\Delta g(t_{n-2}))\|\Delta g(t_{n-1})\|^2}>0.
$$
\end{definition}
Evidently, a strongly locally nondeterministic process is locally nondeterministic. However, the reverse statement is not true in general.

Let us assume  $K(t,s):=K(t-s),\ t,s\in[0,1]$ and that it is a continuous function. The following lemma about general centred Gaussian processes is immediate, and will be useful to study renormalized self-intersection local times of Volterra Gaussian processes.

\begin{lemma}
\label{lem:g_str_loc_nondet} Let  $g(t)(s) \coloneq K(t-s)1_{[0,t]}(s),\ t,s\in[0,1],$ where $K$ is  continuous, continuously differentiable with respect to $t$ with $K(0)\neq 0$,  then  the function $g$ is strongly locally nondeterministic in the sense of Definition~\ref{def:strongLND}.
\end{lemma}
\begin{proof}[Proof of Lemma~\ref{lem:g_str_loc_nondet}] Let $\xi$ be a white noise in $L_2([0,1])$ generated by the Wiener process $W(t),\ t\in[0,1].$ Consider a Gaussian process 
$$
x(t)=\inner{g(t)}{\xi},\ t\in[0,1].
$$
 In order to prove the lemma by the method of induction, it suffices to check that
$$
G(\Delta g(t_1),\ldots,\Delta g(t_{k-1})) \sim\prod^{k-1}_{i=1}\|\Delta g(t_i)\|^2,\ \text{as}\ (t_{k}-t_1) \to 0.
$$
Note that
\begin{align*}
  G(\Delta g(t_1),\ldots,\Delta g(t_{k-1}))&=\Var(\Delta x(t_1)) \prod_{i=1}^{k-2}\Var(\Delta x(t_{i+1})\mid \Delta x(t_1), \Delta x(t_2), \ldots,\Delta x(t_i)),
\end{align*}
where $\Delta x(t_i) \coloneq x(t_{i+1})-x(t_i)$.
Hence, we have
\begin{align}
 G(\Delta g(t_1),\ldots,\Delta g(t_{k-1}))&{}\leqslant \prod^{k-1}_{i=1}\Var(\Delta x(t_i)) \nonumber \\
 &{}=\prod^{k-1}_{i=1}\|\Delta g(t_i)\|^2 \nonumber \\
  &{} = \prod^{k-1}_{i=1}\left(\int^{t_{i+1}}_{t_i}K^2(t_{i+1}-r)\differential{r}+\int^{t_i}_0(K(t_{i+1}-r)-K(t_i-r))^2\differential{r}\right), \label{eq:G_delta_g_upper}
\end{align}
and in the other direction, 
\begin{align}
  G(\Delta g(t_1),\ldots,\Delta g(t_{k-1}))&{}\geqslant \prod^{k-1}_{i=1}\Var(\Delta x(t_i)| W(t_1), W(t_2),\ldots, W(t_i)) \nonumber  \\
  &{} = \prod^{k-1}_{i=1}\int^{t_{i+1}}_{t_i}K^2(t_{i+1}-r)\differential{r}. \label{eq:G_delta_g_lower}
\end{align}

 It follows from  \eqref{eq:G_delta_g_upper} and \eqref{eq:G_delta_g_lower}  that
 $$
 G(\Delta g(t_1),\ldots,\Delta g(t_{k-1}))\sim K^{2(k-1)}(0)\prod^{k-1}_{i=1}(t_{i+1}-t_i),
  $$
as  $(t_k-t_1) \to 0.$
 Since we have 
 $$
 \prod^{k-1}_{i=1}\|\Delta g(t_i)\|^2\sim K^{2(k-1)}(0)\prod^{k-1}_{i=1}(t_{k+1}-t_k),
  $$
as  $(t_k-t_1)\to 0,$ we immediately conclude that 
  $$
  G(\Delta g(t_1),\ldots,\Delta g(t_{k-1}))\sim\prod^{k-1}_{i=1}\|\Delta g(t_i)\|^2, 
  $$
as $(t_k-t_1) \to 0$,  which completes the proof.
\end{proof}

We are now ready to construct the Rosen renormalized self-intersection local time for a planar Volterra Gaussian process. Let $\xi_1,\ \xi_2$ be two independent white noises in $L_2([0,1])$ and $g(t)(\cdot)\coloneq K(t-\cdot)1_{[0,t]}(\cdot),\ t\in[0,1].$  Consider a planar Volterra Gaussian process defined as follows
$$
Y(t)\coloneq (Y_1(t), Y_2(t))=(\inner{g(t)}{\xi_1},\inner{g(t)}{\xi_2}),\ t\in[0,1].
$$
Consider the formal expression for the $k$-multiple self-intersection local time of $Y$ as
\begin{align*}
T_{k}^{Y}\coloneq \int_{\Delta_k}\left(\prod^{k-1}_{i=1}\delta_{0}(Y(t_{i+1})-Y(t_i))\right) \prod_{i=1}^{k}\differential{t_i},
\end{align*}
and its approximation
$$
T_{\ve, k}^{Y}\coloneq \int_{\Delta_k}\left(\prod^{k-1}_{i=1}f_{\ve}(Y(t_{i+1})-Y(t_i))\right) \prod_{i=1}^{k}\differential{t_i}.
$$
In order to present the formal Fourier--Wiener transform of $T_{\ve, k}^{Y}$, we need the following technical lemma proved in \cite{Dorogovtsev2011regularization}. Denote by $A(e_1,\ldots,e_n)$ the Gram matrix constructed by the elements  $e_1,\ldots,e_n$ of a Hilbert space $H$ and by $P_{e_1,\ldots,e_n},$ the projection in $H$ onto the linear subspace generated by the elements $e_1,\ldots,e_n.$ 

 \begin{lemma}
\label{lem:projection} Let $e_1,\ldots, e_n$ be a collection of linearly independent elements of a Hilbert space $H.$ Then, for any $h\in H$ and
$u=(\inner{e_1}{h},\ldots,\inner{e_n}{h})$, the following relation holds
$$
\inner{A^{-1}({e_1,\ldots,e_n})u}{u}=\|P_{e_1,\ldots,e_n}h\|^2.
$$
\end{lemma}
We omit the proof of Lemma~\ref{lem:projection} as it can be proved using similar arguments as in \cite[Lemma~1]{Dorogovtsev2011regularization}. 

Let $\alpha$ be a square integrable random variable measurable with respect to white noises $\xi_1$ and $\xi_2$. Let us denote the Fourier--Wiener transform of $\alpha$ by $\mathcal{T}(\alpha).$ See Definition~\ref{def:fourier_wiener} in Appendix~\ref{sec:white_noise}. Then, for $h_1, h_2 \in H$,  we can define the Fourier--Wiener transform of $\prod^{k-1}_{i=1}\delta_{0}(Y(t_{i+1})-Y(t_i))$ in $T_{k}^{Y}$ as the limit of the Fourier--Wiener transforms of the approximating products $\prod^{k-1}_{i=1}f_{\ve}(Y(t_{i+1})-Y(t_i))$ in $T_{\ve, k}^{Y}$ as $\ve\to0$:
\begin{align*}
  &\mathcal{T}\Big(\prod^{k-1}_{i=1}\delta_0(Y(t_{i+1})-Y(t_i))\Big)(h_1,h_2)\\
  &{}=\lim_{\ve\to0}\mathcal{T}\Big(\prod^{k-1}_{i=1}f_{\ve}(Y(t_{i+1})-Y(t_i))\Big)(h_1,h_2)\\
  &{} = \frac{1}{(2\pi)^{k-1}\det \Cov\left(\Delta Y_1(t_1),\ldots,\Delta Y_1(t_{k-1})\right)}\myExp{-\frac{1}{2}\left(\inner{A^{-1}_{t_1,\ldots,t_k}u_1}{u_1}+\inner{A^{-1}_{t_1,\ldots,t_k}u_2}{u_2}\right)},
\end{align*}
where we use notations
\begin{align*}
  A_{t_1,\ldots,t_k}\coloneq A(\Delta g(t_1),\ldots,\Delta g(t_{k-1})),\quad P_{t_1,\ldots,t_k}\coloneq P_{\Delta g(t_1),\ldots,\Delta g(t_{k-1})}, 
\end{align*}
and $$u_i=(\inner{\Delta g(t_1)}{h_i},\ldots,\inner{\Delta g(t_{k-1})}{h_i}),\ i=1,2.$$
Applying Lemma~\ref{lem:projection}, we conclude that the formal Fourier--Wiener transform of 
$T_{k}^{Y}$ has the following form
\begin{equation}
\label{eq4.3}
\int_{\Delta_k}\frac{1}{(2\pi)^{k-1}G(\Delta g(t_1),\ldots,\Delta g(t_{k-1}))}\myExp{-\frac{1}{2}(\|P_{t_1\ldots t_k}h_1\|^2+\|P_{t_1\ldots t_k}h_2\|^2)}\prod_{i=1}^{k}\differential{t_i}.
\end{equation}
Unfortunately, the integral \eqref{eq4.3} diverges. Therefore, we need to construct a regularization of this divergent integral, which we do in the next section. 

\subsection{Construction of the regularization}
\label{sec:regularization}
In order to intuit a regularization for the planar Volterra Gaussian process $Y$, we first consider the case of a planar Wiener process
$$
\mathcal{W}(t)=( \inner{1_{[0,t]}}{\xi_1}, \inner{1_{[0,t]}}{\xi_2}),\ t\in[0,1].
$$ 
Note that for a planar Wiener process $\mathcal{W}$, the functions $\Delta g(t_i)=1_{[t_i,t_{i+1}]}$, for  $i=1,\ldots, k-1$ are orthogonal. It implies that the formal Fourier--Wiener transform of the Rosen renormalized self-intersection local time of planar Wiener process has the form
$$
\int_{\Delta_k}\frac{1}{(2\pi)^{k-1}\prod^{k-1}_{i=1}(t_{i+1}-t_i)}\prod^{k-1}_{i=1} \left(\myExp{-\frac{1}{2}(\|P_{t_{i}t_{i+1}}h_1\|^2+\|P_{t_{i}t_{i+1}}h_2\|^2)}-1\right)\prod_{i=1}^{k}\differential{t_i}.
$$
 See \cite{Dorogovtsev2011regularization} for more details. 


Let us return to the case of a  planar Volterra Gaussian process $Y$. The following statement describes the Rosen renormalized Fourier--Wiener transform of the $k$-multiple self-intersection local time for the planar Volterra Gaussian process $Y$.

\begin{theorem}
\label{thm:Fourier_Wiener} Consider $g(t)(\cdot)=K(t-\cdot)1_{[0,t]}(\cdot),\ t\in[0,1],$ where $K$ is a continuously differentiable function on $[0,1]$ with
$K(0)\neq 0.$ Let $\tilde{\Delta} g(t_1),\ldots, \tilde{\Delta} g(t_{k-1})$ be the orthogonal system obtained from $\Delta g(t_1),\ldots, \Delta g(t_{k-1})$ via the Gram--Schmidt orthogonalization procedure. Then, the Rosen renormalized self-intersection local time of multiplicity $k$ exists for any $k\geqslant 2$, i.e.,  the following integral 
$$
\int_{\Delta_k}\prod^{k-1}_{i=1}\frac{1}{\|\tilde{\Delta} g(t_i)\|^2}\left(\myExp{-\frac{1}{2}(\|P_{\tilde{\Delta} g(t_i)} h_1\|^2+\|P_{\tilde{\Delta} g(t_i)} h_2\|^2)}-1\right)\prod_{i=1}^{k}\differential{t_i}
$$
converges for any $h_1,h_2\in L_2([0,1])$. 

\end{theorem}
\begin{proof}[Proof of Theorem~\ref{thm:Fourier_Wiener}] It suffices to check that for any $h\in L_2([0,1])$ the following integral 
\begin{align}
  \mathcal{I}=&{}\int_{\Delta_k}\prod^{k-1}_{i=1}\frac{1}{\|\tilde{\Delta} g(t_i)\|^2}\left(1-\myExp{-\frac{1}{2\|\tilde{\Delta} g(t_i)\|^2}\inner{h}{\tilde{\Delta} g(t_i)}^2}\right)\prod_{i=1}^{k}\differential{t_i} \label{eq4.5}
\end{align}
converges. 
Note that
\begin{align*}
  \mathcal{I} &{} \leqslant c\int_{\Delta_k}\prod^{k-1}_{i=1}\frac{\inner{ h }{\tilde{\Delta} g(t_i)}^2}{\|\tilde{\Delta} g(t_i)\|^4}\prod_{i=1}^{k}\differential{t_i},
\end{align*}
where $c$ is a positive constant. Now,  let us first integrate with respect to $t_k.$ Consider the integral 
$$
\int^1_{t_{k-1}}\frac{\inner{h}{\tilde{\Delta} g(t_{k-1})}^2}{\|\tilde{\Delta} g(t_{k-1})\|^4}\differential{t_k}.
$$
 Note that we have 
 $$
 \tilde{\Delta} g(t_{k-1})=\Delta g(t_{k-1})-P_{t_1\ldots t_{k-1}}\Delta g(t_{k-1}).
 $$
It follows from Lemma~\ref{lem:projection} that
 \begin{align*}
  \norm{\Delta g(t_{k-1})-P_{t_1\ldots t_{k-1}}\Delta g(t_{k-1})}^2 &= \frac{G(\Delta g(t_1), \ldots, \Delta g(t_{k-1}))} {G(\Delta g(t_1), \ldots, \Delta g(t_{k-2}))}\sim K^2(0)(t_{k}-t_{k-1}),
\end{align*}
as $(t_k-t_1) \to 0.$ Therefore, we have 
\begin{align*}
  \int^1_{t_{k-1}}
\frac{\inner{h}{\tilde{\Delta} g(t_{k-1})}^2 }{\|\tilde{\Delta} g(t_{k-1})\|^4}\differential{t_k}&{}\leq c_1\int^1_{t_{k-1}}
\frac{\inner{h}{\Delta g(t_{k-1})-P_{t_1\ldots t_{k-1}}\Delta g(t_{k-1})}^2
}
{(t_{k}-t_{k-1})^2}\differential{t_k}\\
&{} \leq 2c_1\left(\int^1_{t_{k-1}}
\frac{
  \inner{ h}{ \Delta g(t_{k-1})}^2
}
{(t_{k}-t_{k-1})^2}\differential{t_k} \right.\\
&{}\quad\quad\quad   \left. +
\int^{1}_{t_{k-1}}
\frac{\inner{h}{P_{t_1\ldots t_{k-1}}\Delta g(t_{k-1})}^2
}
{(t_{k}-t_{k-1})^2}\differential{t_k}
\right),
\end{align*}
for some $c_1>0$. 
%
Now, assume that $h\geq0.$  Let us check that there exists a positive constant $c_2$ such that
$$
\int^1_{t_{k-1}}
\frac{\inner{h}{\Delta g(t_{k-1})}^2
}
{(t_{k}-t_{k-1})^2}\differential{t_k}\leq c_2\|h\|^2.
$$
Now, note that
\begin{align*}
  &\int^1_{t_{k-1}} \frac{\inner{h}{\Delta g(t_{k-1})}^2}{(t_k-t_{k-1})^2}\differential{t_k} \\
  &{} = \int^1_{t_{k-1}}\frac{\left(\int^{t_k}_{t_{k-1}}h(s)K(t_k,s)\differential{s}+\int^{t_{k-1}}_0h(s)(K(t_k,s)-K(t_{k-1},s))\differential{s}\right)^2}{(t_k-t_{k-1})^2}\differential{t_k} \\
&{}
\leqslant 2\int^1_{t_{k-1}}
  \left(\int^{t_{k}}_{t_{k-1}}
  h(s)K(t_k,s)\differential{s}\right)^2
  \frac{1}{(t_k-t_{k-1})^2}\differential{t_k} \\
   &{}\quad 
  +2\int^1_{t_{k-1}}
  \left(\int^{t_{k-1}}_{0}
  h(s)(K(t_k,s)-K(t_{k-1},s))\differential{s}\right)^2
  \frac{1}{(t_k-t_{k-1})^2}\differential{t_k}.
\end{align*}
It is not difficult to see that there exists a positive constant $c_3$ such that
\begin{align*}
  \int^1_{t_{k-1}}\left(\int^{t_{k-1}}_{0}h(s)(K(t_k,s)-K(t_{k-1},s))\differential{s}\right)^2\frac{1}{(t_k-t_{k-1})^2}\differential{t_k}\leqslant c_3\|h\|^2.
\end{align*}
Now, let us check that the same estimate holds for the first summand. Indeed, we have 
\begin{align*}
  & \int^1_{t_{k-1}} \int^{t_k}_{t_{k-1}}\int^{t_k}_{t_{k-1}}h(s_1)h(s_2)K(t_k,s_1)K(t_k,s_2)\differential{s_1}\differential{s_2} \frac{1}{(t_k-t_{k-1})^2}\differential{t_k}\\
  &{} \leqslant c_4 \int^{1}_{t_{k-1}}\int^{1}_{t_{k-1}}
  h_1(s_1)h_1(s_2)
  \int^1_{s_1\vee s_2}
  \frac{1}{(t_k-t_{k-1})^2}\differential{t_k}
  \differential{s_1}\differential{s_2} \\
  &{} \leqslant c_4\int^1_{t_{k-1}}\int^1_{t_{k-1}}
  h_1(s_1)h_1(s_2)
  \frac{1}{s_1\vee s_2-t_{k-1}}
  \differential{s_1}\differential{s_2}\\
  &{} =
  2c_4\int^1_{t_{k-1}}h_1(s_1)\int^1_{s_1}
  \frac{h_1(s_2)}{s_2-t_{k-1}}\differential{s_2}\differential{s_1}.  
\end{align*}
Here $a\vee b$ denotes $\max\{a,b\}.$ Consider in $L_2([t_{k-1}, 1])$  the integral operator with the kernel 
$$
\mathcal{K}(s_1, s_2)=\frac{1}{s_2-t_{k-1}}1_{\{s_2>s_1\}}.
$$
Using the Shur test \cite{halmosh1987bounded} (see Lemma~\ref{lem:Shur_test} in Appendix~\ref{appendix:extra}) we verify that $\mathcal{K}$  defines a bounded operator in  $L_2([t_{k-1}, 1])$.
To that end, put 
$$
p(s_1)=\frac{1}{\sqrt{s_1-t_{k-1}}},\quad \text{ and }\quad
q(s_2)=\frac{1}{\sqrt{s_2-t_{k-1}}}.
$$
Then, we have 
$$
\int^1_{t_{k-1}}\mathcal{K}(s_1, s_2)q(s_2)\differential{s_2} =\int^1_{s_1}
\frac{1}{(s_2-t_{k-1})^{3/2}}\differential{s_2}\leq
2\frac{1}{\sqrt{s_1-t_{k-1}}}, 
$$
and
$$
\int^1_{t_{k-1}}
\mathcal{K}(s_1, s_2)p(s_1)\differential{s_1} = \frac{1}{s_2-t_{k-1}}\int^{s_2}_{t_{k-1}}
\frac{1}{\sqrt{s_1-t_{k-1}}}\differential{s_1}
=
\frac{2}{\sqrt{s_2-t_{k-1}}}.
$$
Hence, by virtue of the Shur test (Lemma~\ref{lem:Shur_test}), we have
$$
\int^1_{t_{k-1}}h_1(s_1)\int^1_{s_1}
\frac{h_1(s_2)}{s_2-t_{k-1}}\differential{s_2}\differential{s_1}\leq 4\|h\|^2.
$$
Using  similar arguments, we conclude that
\begin{align*}
 &{}\int^{t_k}_{t_{k-1}}
\frac{\inner{h}{P_{t_1\ldots t_{k-1}}\Delta g(t_{k-1})}^2}
{(t_{k}-t_{k-1})^2}\differential{t_k}\\
 &{}=\int^{t_k}_{t_{k-1}}
\frac{\inner{P_{t_1\ldots t_{k-1}}h}{\Delta g(t_{k-1})}^2}
{(t_{k}-t_{k-1})^2}\differential{t_k}\leqslant c_5\|h\|^2,\ c_5>0.
\end{align*}
Hence, there exists a positive constant $c_2$ such that
$$
\int^1_{t_{k-1}}
\frac{\inner{h}{\Delta g(t_{k-1})}^2
}
{(t_{k}-t_{k-1})^2}\differential{t_k}\leq c_2\|h\|^2.
$$
Using similar arguments, we deduce that
$$
\int_{\Delta_k}
\prod^{k-1}_{i=1}
\frac{\inner{h}{\tilde{\Delta} g(t_i)}^2}
{t_{i+1}-t_i} \prod_{i=1}^{k} \differential{t_i} \leq c_6\|h\|^{2k},\ c_6>0,
$$
which completes the proof of the theorem.
\end{proof}
We are now ready to present the It\^o--Wiener expansion for the Rosen renormalized $k$ multiple self-intersection local time of planar Volterra Gaussian process. We need a few new notations, which we introduce below. For a function $f:[0,1]\to\mathbb{R},$ denote by
$$
f^{\bigotimes k}(s_1,\ldots,s_k)=\prod^{k}_{i=1}f(s_i),\ s_1,\ldots,s_k\in [0,1].
$$
For nonnegative integers $n, k$, define the sets 
\begin{align*}
  \Lambda({n, k}) &\coloneqq \{m = (m_1, m_2, \ldots, m_k) \in \mbN_0^k \mid \sum_{i=1}^{k} m_i=n  \}, \\
  \Lambda_{+}({n, k}) &\coloneqq \{m = (m_1, m_2, \ldots, m_k) \in \mbN^k \mid  \sum_{i=1}^{k} m_i=n  \}. 
\end{align*}

 \begin{lemma}
\label{lem:It\^o--Wiener expansion} The It\^o--Wiener expansion for the Rosen renormalized $k$-multiple self-intersection local time of the planar Volterra Gaussian process $Y$ has the following representation
\begin{align*}
&{}\tilde{T}^{Y}_k=\sum^{\infty}_{n=k-1}\sum^{n}_{p=0}\int_{\Delta_{2n}}\alpha_{2n}(s_1,\ldots,s_{2n})\differential{W_1(s_1)}\ldots \differential{W_1(s_{2p})}\differential{W_2(s_{2p+1})}\ldots\differential{W_2(s_{2n})},
 \end{align*}
 where the coefficients $\alpha_{2n}(s_1,\ldots,s_{2n})$ are given by
 \begin{align*}
  &{}\alpha_{2n}(s_1,\ldots,s_{2n}) 
  \coloneqq \frac{(2n)!}{(2\pi)^{k-1}}\left(-\frac{1}{2}\right)^{n}\sum_{m \in \Lambda_{+}(n-1, k-1)}\sum_{l\in \Lambda(p, k-1)}  \beta_{2n}^{m, l}(s_1, s_2, \ldots, s_{2n}) \\
\end{align*}
with  
\begin{align*}
  &\beta_{2n}^{m, l}(s_1, s_2, \ldots, s_{2n}) \\ &{} \coloneqq  \int_{\Delta_k}\prod^{k-1}_{i=1} \left[ \binom{m_i}{l_i}(\tilde{\Delta}g(t_i))^{\bigotimes 2l_i}(s_1,\ldots,s_{2l_i})(\tilde{\Delta}g(t_i))^{\bigotimes 2m_i-2l_i}(s_{2l_i+1},\ldots,s_{2m_i-2l_i}) \right.\\ 
  &{}\quad\quad\quad  \left. \times 
  \frac{1}{ m_i! \,  \|\tilde{\Delta}g(t_i)\|^{2n+2}  }\right] \prod^{k}_{i=1}\differential{t_i},
\end{align*}
for $m = (m_1, m_2, \ldots, m_{k-1}) \in \Lambda_{+}(n-1, k-1)$ and $l = (l_1, l_2, \ldots, l_{k-1}) \in \Lambda(p, k-1)$.

  \end{lemma}
 \begin{proof}[Proof of Lemma~\ref{lem:It\^o--Wiener expansion}] 
 It follows from  Theorem~\ref{thm:Fourier_Wiener} that the  Fourier--Wiener transform of the Rosen renormalized $k$-multiple self-intersection local time of $Y$ has the  representation
\begin{align*}
 &{}  \int_{\Delta_k} \frac{1 }{(2\pi)^{k-1}} \prod^{k-1}_{i=1}\frac{1}{\|\tilde{\Delta}g(t_i)\|^2}\left(\exp\left(-\frac{1}{2}\left(\|P_{\tilde{\Delta}g(t_i)}h_1\|^2+\|P_{\tilde{\Delta}g(t_i)}h_2\|^2\right)\right)-1\right)\prod^k_{i=1}\differential{t_i}\\
 &{} =\int_{\Delta_k}\frac{1}{(2\pi)^{k-1}}\prod^{k-1}_{i=1}\left[\sum^{\infty}_{m=1}\frac{1}{m!}\left(-\frac{1}{2}\right)^m\left(\|P_{\tilde{\Delta}g(t_i)}h_1\|^2+\|P_{\tilde{\Delta}g(t_i)}h_2\|^2\right)^m\frac{1}{\|\tilde{\Delta}g(t_i)\|^2}\right] \prod^{k}_{i=1}\differential{t_i}\\
 &{} =\int_{\Delta_k}\frac{1}{(2\pi)^{k-1}}\prod^{k-1}_{i=1} \left[  \sum^{\infty}_{m=1}\frac{1}{m!}\left(-\frac{1}{2}\right)^m\sum^{m}_{l=0}\binom{m}{l}(P_{\tilde{\Delta}g(t_i)}h_1\|)^{2l}(\|P_{\tilde{\Delta}g(t_i)}h_2\|)^{2(m-l)} \right.\\
 &{}\quad \quad \left. \times \frac{1}{\|\tilde{\Delta}g(t_i)\|^2}\right] \prod^{k}_{i=1}\differential{t_i}\\
 &{} =\int_{\Delta_k}\frac{1}{(2\pi)^{k-1}}\sum^{\infty}_{m_1,\ldots,m_{k-1}=1}\frac{1}{m_1!\ldots m_{k-1}!}\left(-\frac{1}{2}\right)^{m_1+\ldots +m_{k-1}}\\
 &{} \quad  \sum_{0\leqslant l_1\leqslant m_1,\ldots,\ 0\leqslant l_{k-1}\leqslant m_{k-1}}\prod^{k-1}_{i=1}\binom{m_i}{l_i}\|P_{\tilde{\Delta}g(t_i)}h_1\|^{2l_i}\|P_{\tilde{\Delta}g(t_i)}h_2\|^{2m_i-2l_i}\frac{1}{\|\tilde{\Delta}g(t_i)\|^2}\prod^{k}_{i=1}\differential{t_i}\\
 &{}=\sum^{\infty}_{n=k-1}\int^1_0 \overset{2n}{\cdots}\int^1_0\int_{\Delta_k}\frac{1}{(2\pi)^{k-1}}\left(-\frac{1}{2}\right)^{n}\sum^{\infty}_{m_1+\ldots+m_{k-1}=n, m_1>0,\ldots,m_{k-1}>0}\frac{1}{m_1!\ldots m_{k-1}!}\\
&{}\quad \times  \sum^{n}_{p=0}\sum_{l_1+\ldots+l_{k-1}=p}\prod^{k-1}_{i=1}\binom{m_i}{l_i}\\ 
&{}\quad \times \int_{\Delta_k}\prod^{k-1}_{i=1}(\tilde{\Delta}g(t_i))^{\bigotimes 2l_i}(s_1,\ldots,s_{2l_i}) (\tilde{\Delta}g(t_i))^{\bigotimes 2m_i-2l_i}(s_{2l_i+1},\ldots,s_{2m_i-2l_i})\\
&{}\quad \times  h_1^{\bigotimes 2p}(s_1,\ldots,s_{2p}) h_2^{\bigotimes 2n-2p}(s_{2p+1},\ldots,s_{2n})\frac{1}{\|\tilde{\Delta}g(t_i)\|^{2n+2}}\prod^{k}_{i=1}\differential{t_i}\prod^{2n}_{j=1}\differential{s_j}.
\end{align*}
Hence, we get 
\begin{align*}
\tilde{T}^{Y}_k &{} = \sum^{\infty}_{n=k-1}\int^1_0\overset{2n}{\cdots}\int^1_0\frac{1}{(2\pi)^{k-1}}\left(-\frac{1}{2}\right)^{n}\sum^{\infty}_{m_1+\ldots+m_{k-1}=n-1, m_1>0,\ldots,m_{k-1}>0}\frac{1}{m_1!\ldots m_{k-1}!}\\
&{}\quad \times  \sum^{n}_{p=0} \sum_{l_1+\ldots+l_{k-1}=p}\prod^{k}_{i=1}\binom{m_i}{l_i}\\
 &{}\quad \times \int_{\Delta_k} \left[\prod^{k-1}_{i=1}(\tilde{\Delta}g(t_i))^{\bigotimes 2l_i}(s_1,\ldots,s_{2l_i})(\tilde{\Delta}g(t_i))^{\bigotimes 2m_i-2l_i}(s_{2l_i+1},\ldots,s_{2m_i-2l_i}) \right.\\ 
&{}\quad \left. \times \frac{1}{\|\tilde{\Delta}g(t_i)\|^{2n+2}}\prod^{k}_{i=1}\differential{t_i} \right]\,\differential{W_1(s_1)}\ldots \differential{W_1(s_{2p})}\differential{W_2(s_{2p+1})}\ldots\differential{W_2(s_{2n})}.
 \end{align*}
 Rearranging the terms, we obtain the desired result.
 
   \end{proof}

\section{Conclusions}
\label{sec:conclusions}

The law of the iterated logarithm for a Brownian motion was extended  to a certain class of Gaussian processes by H. Oodaira in \cite{Oodaira1972Strassen,Oodaira1973Law} and M. A. Arcones in   \cite{arcones1995law}. If $x(t),\ t\geqslant 0$ is a Gaussian process with the covariance $R(t,s),\ t,s\geqslant 0,$ then in \cite{Oodaira1972Strassen}   the law of the iterated logarithm was proved for a class of semi-stable Gaussian processes, i.e.,  a class of Gaussian processes for which there exists a positive function $v(r),\ r\geqslant 0$ such that $v(r)$ converges to infinity as $r\to\infty$ and
$$
R(rt,rs)=v(r)R(t,s),\ r,\ t,\ s\geqslant 0.
$$
If such a function $v$ exists, then it must be of the form $r^{\alpha},\ \alpha>0.$ Semi-stable Gaussian processes were studied by Lampertri in \cite{lamperti1962semi}. In \cite{Oodaira1973Law}, the semi-stability condition was replaced by the following asymptotic one:  there are a positive function $v(r),\ r\geqslant 0$ and a covariance kernel $\Gamma(t,s),\ t,\ s\geqslant 0$ such that 
$$
v^{-1}(r)R(rt,rs)\to\Gamma(t,s)\ \text{as}\ r\to\infty\ \text{uniformly in}\ t,\ s\in[0,1].
$$
An example of a not semi-stable process that satisfies the asymptotic condition is
$$
x(t)=\int^t_0Z(s)\differential{s},
$$
where $Z$ is a Ornstein--Uhlenbeck process. In \cite{arcones1995law}, the law of the iterated logarithm was established for a class of self-similar Gaussian processes, in particular, for a fractional Brownian motion with Hurst parameter $\frac{1}{2}\leqslant H<1.$ The law of the iterated logarithm for a fractional Brownian motion with Hurst parameter $0<H\leqslant\frac{1}{2}$ was proved in \cite{jiawei2018fine}. In this paper, we prove the law of the iterated logarithm for a class of Volterra Gaussian processes that, in general, do not satisfy any of mentioned conditions. 
 
 In the second part of the paper, we discuss local times and self-intersection local times that in some sense contain  information about nondeterminism of sample paths. Joint space-time regularity of the local time for a fractional Brownian motion was studied in \cite{butkovsky2023stochastic} and for Volterra--Levy processes was studied in \cite{harang2022regularity}. The regularity of the local time for stochastic equations with smooth coefficients driven by the fractional Brownian motion was considered in \cite{lou2017local}. Joint space time regularity of the occupation measure and self-intersection measure for a general Volterra--It\^o processes was studied in \cite{friesen2024regular}.
 In the mentioned works, authors give a local nondeterminism condition applicable to the class of processes that allow one to obtain the space time regularity for either the occupation measure or the self-intersection measure. The essential difference with the class of Volterra Gaussian processes considered in this paper that it does not have a local nondeterminism  property in dimensions $d\geqslant 2.$ Under some conditions on the kernel that generate a planar Volterra Gaussian process,  it behaves like a planar Wiener process on small time intervals. This implies that the self-intersection local times do not exist, and, as in the case of planar Wiener process, we must construct a renormalization. Using a white noise approach we construct the Rosen renormalized self-intersection local time for a planar Volterra Gaussian process.

\appendix 
  \section{White noise. It\^o--Wiener expansion and Fourier--Wiener transform}
  \label{sec:white_noise}
  Let us define a white noise in a real separable Hilbert space $H$ with an inner product $\inner{\cdot}{\cdot}.$ All Hilbert spaces in this paper are assumed to be real. We refer the readers to \cite{Janson1997GHS} 
  for more details. 
  \begin{definition}
  \label{def:white_noise} 
  A random functional $\xi$ on $H$  such that 
  \begin{enumerate}
  \item
  the mapping $h\to \xi(h)$ is linear
  \item
  $E\xi(h)=0\  \text{for all}\  h\in H$
  \item
   $E\xi(h)\xi(g)=\inner{h}{g}\  \text{for all}\  h,\ g\in H$
   \end{enumerate}
   is said to be a white noise in $H.$
   \end{definition}
   In the case $H=\mbR^n,\ n\geqslant 1,$ the white noise is $\xi\sim\mathcal{N}(0,I)$ and $\xi(h)=\inner{h}{\xi},\ h\in H,$ i.e., the random functional is generated by a random element in $\mbR^n$. The same is not true in case of an infinite-dimensional Hilbert space. The white noise in an infinite-dimensional separable Hilbert space does not correspond to any random element in this space. Indeed, let $H$ be a separable Hilbert space with $\dim H=\infty$ and  $\xi_n,\ n\geqslant 1$ be a sequence of independent standard Gaussian random variables. Consider an orthonormal basis $\{e_n,\ n\geqslant 1\}$ in H. Define
  $$
  \xi(h)=\sum^{\infty}_{n=1}\inner{h}{e_n}\xi_n. 
  $$
  Applying Kolmogorov's two-series theorem \cite[Theorem 1.4.2, page 37]{Stroock1993Probability}, one can verify that the series converges almost surely. It is not difficult to see that $\xi(h),\ h\in H$ is a white noise according to Definition~\ref{def:white_noise}. Assume that there exists a random element $\xi^{\prime}$ in H such that for any $h\in H$ we have $ \xi(h)=\inner{h}{\xi^{\prime}}$ almost surely. Then, 
  $$
  \norm{\xi^{\prime}}^2 \coloneq \inner{\xi^{\prime} }{\xi^{\prime}}=\sum^{\infty}_{n=1}\inner{\xi^{\prime}}{e_n}^2=\sum^{\infty}_{n=1}\xi(e_n)^2<\infty\ \text{a.s.}
  $$
  but 
  $$
  \sum^{\infty}_{n=1}\xi(e_n)^2=\sum^{\infty}_{n=1}\xi_n^2=\infty\ \text{a.s.}, 
  $$
  which is a contradiction. 
  This shows that a random functional $\xi$ on $H$ cannot be associated with any random element in $H.$ 
  Sometimes a white noise in an infinite-dimensional Hilbert space is called a generalized random element. In this paper, we use the notation $\inner{h}{\xi}$ for $\xi(h).$ Consider the white noise $\xi$ in $L_2([0,1]).$ If $W(t),\ t\in[0,1]$ is a Wiener process, then a random functional on $L_2([0,1])$ is defined as follows
  $$
  \inner{h}{\xi}=\int^1_0h(t)\differential{W(t)},\ h\in L_2([0,1]).
  $$
  Using the properties of stochastic integral, one can check that $\xi$ is a white noise in $L_2([0,1]).$ 
  
  The It\^o--Wiener  and the Fourier--Wiener transforms are useful tools for the investigation of  functionals of Wiener processes. Let $\alpha$ be a  square-integrable random variable. Assume that $\alpha$ is measurable with respect to the planar Wiener process $\mathcal{W}(t)=(W_1(t),W_2(t)),\ t\in[0,1]$.   Let us recall the definitions of the It\^o--Wiener expansion and the Fourier--Wiener transform for the random variable $\alpha.$
  \begin{theorem}[It\^o--Wiener expansion]
  \label{thm:1} 
  The random variable $\alpha$ can be uniquely represented as the convergent in mean square series of orthogonal summands
  \begin{align*} 
&{}\alpha=\sum^{\infty}_{k=0}\sum^{k}_{l=0}\int_{\Delta_k}\alpha_k(r_1,\ldots,r_k)\differential{W_1(r_1)}\ldots\differential{W_1(r_l)}\differential{W_2(r_{l+1})}\ldots\differential{W_1(r_k)},
\end{align*} 
where
$\alpha_0=\Eof{\alpha},$  kernels $\alpha_k,\ k>0$ satisfy the condition
$$
\int_{\Delta_k}\alpha^2_k(r_1,\ldots,r_k)\differential{r_1}\ldots\differential{r_l}\differential{r_{l+1}}\ldots\differential{r_k}<\infty
$$
and
$$
\Eof{\alpha^2}=\sum^{\infty}_{k=0}\sum^{k}_{l=0}\int_{\Delta_k}\alpha^2_k(r_1,\ldots,r_k)\differential{r_1}\ldots\differential{r_l}\differential{r_{l+1}}\ldots\differential{r_k}.
$$
\end{theorem} 
Consider two independent white noises $\xi_1,\ \xi_2\in L_2([0,1])$ generated by Wiener processes $W_1,\ W_2.$
  \begin{definition}[Stochastic exponential]\label{def:stochastic_exponent}
    For $h_1, h_2 \in H$, the stochastic exponential is defined as 
    \begin{align*}
      \mathcal{E}(h_1, h_2) = \exp\left(\inner{h_1}{ \xi_1} + \inner{h_2}{ \xi_2} - \frac{1}{2} \left(\norm{h_1}^2 + \norm{h_2}^2 \right)\right).
    \end{align*}
  \end{definition}
Since $\alpha$ is a square-integrable and measurable with respect to $\xi_1,\ \xi_2,$ then one can consider its Fourier-Wiener transform defined as follows.
  \begin{definition}[Fourier--Wiener transform]\label{def:fourier_wiener}
    The Fourier--Wiener transform of the random  variable $\alpha$ is defined as 
    \begin{align*}
      \mathcal{T}(\alpha)(h_1, h_2) = \Eof{\alpha \mathcal{E}(h_1, h_2)}.
    \end{align*}
    
  \end{definition}

  \section{Gaussian integrators}
  \label{sec:integrators}
   The following statement describes the structure of all integrators. 
  \begin{lemma}[Structure of integrators \cite{Izyumtseva2014Onthelocaltime}]
  \label{lem:1}
  The Gaussian process $x$ is an integrator iff it can be represented as
  \begin{equation}
  \label{eq1.2}
  x(t)=\inner{A1_{[0, t]}}{\xi}, \ t\in[0,1]
  \end{equation}
  for  some continuous linear operator $A$ on $L_2([0,1])$ and a white noise $\xi$ in the same space.
  \end{lemma}
  The proof of 
  Lemma \ref{lem:1} can be found in \cite{Izyumtseva2014Onthelocaltime}.

  \section{Additional results related to Gaussian processes}
  \label{sec:AppendixC}
  \subsection{Tail probabilities of Gaussian extrema.}
  \label{subsec:Gaussian tails}
  \begin{lemma}[Tail estimate for Gaussian extrema (\cite{Samorodnitsky1991tails})]
  \label{lem:6} Assume that $x(t),\ t\in T$ is a zero-mean separable Gaussian process with a parameter set $T.$ Then, for each $\delta>0$ there is a positive constant $k(\delta)<\infty$ such that for any $\lambda>0$, we have 
  \begin{equation}
  \label{eq6.1}
  \prob(\sup_{t\in T}x(t)>\lambda)\leqslant k(\delta) \myExp{-\frac{(1-\delta)\lambda^2}{2\sigma^2}},
  \end{equation}
  where 
  $$
  \sigma^2:=\sup_{t\in T}\Var(x(t)).
  $$
  \end{lemma}

\subsection{Local times}
\label{subsec:local_times}
\begin{lemma}[Rudenko's criterion (\cite{rudenko2006existence})]
  \label{lem:Rudenko} Let $x(t),\ t\in[0,1]$ be an $\mbR$-valued Gaussian process and let $\Cov(x(s),x(t))$ be the covariance matrix corresponding to the Gaussian vector $(x(s),x(t))$. If
  \begin{equation}
  \label{eq6.6} 
  \int^1_0\int^1_0\frac{1}{\left(\det\Cov(x(s),x(t))\right)^{\frac{1}{2}}}\differential{s}\differential{t}<\infty,
  \end{equation}
  then the local time of  $x$ exists at $0$ on $[0,1],$ and hence, the local time of $x$ exists at any $y\in\mbR$ on $[0,1].$  
\end{lemma}

  \section{Additional mathematical results}
  \label{appendix:extra}

  \begin{lemma}[An extension of the Borel--Cantelli lemma \cite{Petrov2002note}]
  \label{lem:GBCL}  
  Let $B_n,\ n\geqslant 1$ be a sequence of events satisfying conditions
  \begin{enumerate}
  \item $\sum^{\infty}_{n=1}\prob(B_n)=\infty$
  \item 
  $
  \prob(B_n\cap B_m)\leqslant C \prob(B_n)\cdot \prob(B_m)
  $
  for all $n,m>N$ such that $n\neq m$ and for some constants $C>1$ and a natural number $N.$
  \end{enumerate}
  Then,
  $$
  \prob(B_n\ \text{i.o.})\geqslant \frac{1}{C}.
  $$
   \end{lemma}

   \begin{lemma}
    \label{lem:Shur_test}[Shur test \cite{halmosh1987bounded}]  Assume that $a$ is a nonnegative kernel on $[0,1]^2$.
    If there exists measurable functions $p,q:[0,1]\to(0,\infty)$ and positive constants $\alpha,\ \beta$ such that
    $$
    \int^{1}_{0}a(s_1,s_2)q(s_2)\differential{s_2}\leqslant \alpha p(s_1) 
    $$
    for almost all $s_1$ and
    $$
    \int^{1}_{0}a(s_1,s_2)p(s_1)\differential{s_1}\leqslant \beta p(s_2)
    $$
    for almost all $s_2$, then $a$ induces a bounded operator $A$ on $L_2([0,1])$ such that
    $$
    \|A\|^2\leqslant\alpha\beta.
    $$
    \end{lemma}


\section*{Acknowledgements}
OI was  supported by the British Academy. WRKB was supported by the Engineering and Physical Sciences Research Council (EPSRC) [grant number EP/Y027795/1] of UK Research and Innovation (UKRI). He would also like to thank the Isaac Newton Institute for Mathematical Sciences (INI), Cambridge (EPSRC [grant number EP/R014604/1]) for the support and hospitality during an INI Retreat programme when part of the work was undertaken. 


\bibliographystyle{amsplain}
\bibliography{refs}

\end{document}